\documentclass[12pt,a4paper,oneside]{amsart}
\usepackage[a4paper]{geometry}
\usepackage{txfonts}
\DeclareMathAlphabet{\mathcal}{OMS}{cmsy}{m}{n} 

\usepackage{amssymb,enumerate,perpage}

\MakePerPage[2]{footnote}

\newtheorem{theorem}{Theorem}
\newtheorem{lemma}{Lemma}[section]
\newtheorem{proposition}{Proposition}[section]

\theoremstyle{definition}
\newtheorem{definition}{Definition}[section]

\numberwithin{equation}{section}

\newcommand{\vertbar}{\>|\>}
\newcommand{\set}[2]{\ensuremath{\{ #1 \vertbar #2 \}}}
\def\liebrack  {\ensuremath{[\,\cdot\, , \cdot\,]}}

\DeclareMathOperator{\ad}{\mathsf{ad}}
\DeclareMathOperator{\Ad}{Ad}
\DeclareMathOperator{\Ann}{Ann}
\DeclareMathOperator{\dcobound}{d}
\DeclareMathOperator{\End}{End}
\DeclareMathOperator{\id}{id}
\DeclareMathOperator{\Ker}{Ker}

\begin{document}

\title{Ado Theorem for nilpotent Hom-Lie algebras}

\author{Abdenacer Makhlouf}
\address{Universit\'e de Haute-Alsace, Mulhouse, France}
\email{abdenacer.makhlouf@uha.fr}

\author{Pasha Zusmanovich\,$^\dagger$}
\address{University of Ostrava, Ostrava, Czech Republic}
\email{pasha.zusmanovich@osu.cz}

\date{First written July 26, 2018; last revised July 7, 2019}
\thanks{Intern. J. Algebra Comp., to appear; arXiv:1807.10944}
\thanks{$^\dagger$\,Corresponding author}
\keywords{
Hom-Lie algebra; Hom-associative algebra; nilpotent algebra;
faithful representation}
\subjclass[2010]{17A99; 16G99; 17B01; 17B10; 17B30} 

\begin{abstract}
We prove an analog of the Ado theorem -- the existence of a finite-dimensional 
faithful representation -- for a certain kind of finite-dimensional nilpotent 
Hom-Lie algebras.
\end{abstract}

\maketitle

\section{Introduction}

During the last decade, a new class of algebras has been studied extensively in
the literature: the so-called Hom-Lie algebras, anticommutative algebras with 
the Jacobi identity ``twisted'' by an endomorphism (for a precise definition, 
see \S \ref{sec-p}). The impetus came from mathematical physics, where 
particular instances of such structures appeared as symmetries of some physical
models, and as $\sigma$-deformations. As it happens, after that people started 
to study Hom-Lie (and related) algebras on their own. A lot of interesting
results were obtained recently, we mention here just a few: 
structure theory of simple and semisimple Hom-Lie algebras (\cite{chen-han});
a Hom-analog of classification of filiform Lie algebras (\cite{mm});
cohomology and deformation theory of various classes of Hom-algebras \`a la 
Gerstenhaber, with applications to Hom-analogs of the known constructions 
related to the Witt algebra (\cite{arfa,makhlouf-silv}); Hom-Lie structures on 
Kac-Moody algebras (\cite{mz}); Hom-analogs of some interesting identities in alternative algebras 
(\cite{yau-mikheev}).

When one tries to extend Lie-algebraic constructions and results to the Hom-Lie
case, the in\-no\-cen\-t\-ly-looking, at the first glance, twist often leads to
substantial difficulties of combinatorial character, and the theory of Hom-Lie 
algebras appears to be, generally, more difficult than its ``ordinary'' Lie 
analog. For example, presently it is not known whether any Hom-Lie algebra 
embeds into its universal enveloping algebra, or, more generally, whether an 
analog of the Poincar\'e--Birkhoff--Witt theorem for Hom-Lie algebras holds (see a discussion in \cite[\S 2.3]{hellstrom-et-al}; the
Poincar\'e--Birkhoff--Witt theorem is known to be true only in a very 
particular case of involutive Hom-Lie algebras -- i.e. those for which the square of the twist map 
is identity -- see \cite{gzz}).

In this paper we consider in the Hom-Lie setting another classical Lie-algebraic
result, closely connected with the Poincar\'e--Birkhoff--Witt theorem: the Ado 
theorem which claims the existence of a finite-di\-men\-si\-o\-nal faithful 
representation of any finite-dimensional Lie algebra. The classical proof of the
Ado theorem involves universal enveloping algebras. The analogous construction 
in the Hom-Lie setting is, apparently, much more complicated than in the Lie 
case, and presently is not very well understood (see, again, 
\cite{hellstrom-et-al}, and also \cite{caen-goyv} and \cite{hom-hopf}). That is
why we choose to follow another route, presented in \cite{ado}, which avoids 
usage of universal enveloping algebras and stays entirely inside the category of
finite-dimensional algebras. The drawback of this approach that is works only
for nilpotent algebras and in characteristic zero, so, strictly speaking -- and 
from the historical perspective -- our Hom-Lie result is an analog not of the 
Ado theorem, but of the earlier result: the Birkhoff theorem, \cite{birkhoff},
guaranteeing the existence of a finite-dimensional faithful representation of a
\emph{nilpotent} finite-dimensional Lie algebra.

The contents of the paper are as follows. The first four sections are 
preliminary in nature, and treat elementary concepts related to the Hom-Lie 
theory: representations (\S\S \ref{sec-p} and \ref{sec-ass}), nilpotent algebras
(\S \ref{sec-nilp}), and free algebras (\S \ref{sec-free}). Most of this is 
based on a material existing in the literature, or a variation thereof; but 
because the Hom-Lie theory is still in the formative period, we found it 
necessary to fix the basic concepts and the terminology. Among other things, we
define the tensor product in the category of representations of a multiplicative
Hom-Lie algebra (Proposition \ref{f-tens}); prove that, similarly with the Lie 
case, the existence of a faithful representation is equivalent to the embedding
into a Hom-associative algebra (Theorem \ref{th-a}); and prove that two notions
of nilpotency coincide for multiplicative Hom-Lie algebras over algebraically 
closed fields (Proposition \ref{prop-nilp}). Finally, in \S \ref{sec-ado} we 
prove the Ado theorem for a certain class of nilpotent finite-dimensional 
Hom-Lie algebras -- the so-called multiplicative nondegenerate algebras (for the
precise definitions, see \S 1), defined over an algebraically closed field of 
characteristic zero (Theorem \ref{th-ado}). All the restrictions on the 
structure of algebra are stipulated by our chosen approach, and it remains a 
challenging problem to establish the Ado theorem for a broader class of 
finite-dimensional Hom-Lie algebras -- not necessarily nilpotent, or 
multiplicative, or nondegenerate, or defined over an arbitrary ground field.

Throughout the paper, ``algebra'' means a not necessarily Lie, or associative, 
or satisfying any other distinguished identity, algebra. $\End(V)$ denotes the 
vector space of linear maps from a vector space $V$ to itself. Direct sums, 
denoted by $\oplus$, are understood in the category of vector spaces. The ground
field is assumed arbitrary, unless stated explicitly otherwise. $\mathbb N$ 
denotes the set of all positive integers.

\section{Hom-Lie algebras and their representations}\label{sec-p}

In this section we recall the basic definitions and facts related to 
Hom-Lie algebras and their representations; for the most part, we just 
meticulously record all variations of the relevant concepts, the more so there 
is still some ambiguity in that regard in the literature, and divide everything
into small trivial steps. Perhaps, the only result which can be called new here,
if only conditionally, is the construction of the tensor product of 
representations of a Hom-Lie algebra (Proposition \ref{f-tens}), which is a 
straightforward generalization of the Lie case; but it is implicitly contained 
in \cite{hom-hopf} (and, in the regular case, in \cite{caen-goyv}), where it is
proved that the universal enveloping algebra of a multiplicative Hom-Lie algebra
is Hom-Hopf, suitably defined.

\begin{definition}
A pair $(V,\alpha)$, where $V$ is a vector space, and $\alpha: V \to V$ is a 
linear map, is called a \emph{Hom-vector space}, and the map $\alpha$ is called
a \emph{twist map}. A Hom-vector space with a nondegenerate (i.e., having zero 
kernel) twist map is called \emph{nondegenerate}. A \emph{Hom-subspace} of the 
Hom-vector space $(V,\alpha)$ is a vector subspace of $V$ invariant under 
$\alpha$.
\end{definition}

\begin{definition}
A triple $(A,\cdot,\alpha)$, where $(A,\cdot)$ is an algebra, and $(A,\alpha)$
is a Hom-vector space, is called a \emph{Hom-algebra}. If $\alpha$ is a 
homomorphism of the algebra structure, i.e.
$$
\alpha(x \cdot y) = \alpha(x) \cdot \alpha(y)
$$
for any $x,y \in A$, then the Hom-algebra $(A,\cdot,\alpha)$ is called 
\emph{multiplicative}. If $(A,\alpha)$ is a nondegenerate Hom-vector space, then
the Hom-algebra $(A,\cdot,\alpha)$ is called \emph{nondegenerate}.
\end{definition}

In the literature there exists a close notion of a \emph{regular} Hom-algebra,
namely, a Hom-algebra which is multiplicative and whose twist map is invertible.
That is, in the finite-dimensional situation, which is our main concern here,
regular Hom-algebras are exactly those which are, in our terms, multiplicative 
\emph{and} nondegenerate. But when treating free algebras in \S~\ref{sec-free},
it is more convenient to deal with nondegeneracy rather than with invertibility,
for Hom-algebras with a nondegenerate twist map form a quasivariety, while 
Hom-algebras with an invertible twist map do not. (Still, occasionally we need 
to refer to the situation when the twist map is invertible and not merely nondegenerate, as in Lemmas~\ref{f-ass-end} and \ref{f-4}).

Also, we want to have multiplicativity and nondegeneracy as complementary 
notions, rather than later to be included in the former, for two reasons: first,
it better corresponds to the similar notions for representations (see 
Definitions~\ref{def-repr} and \ref{def-ass}), and, second, the desire to 
clearly isolate places where we need that or another restriction on the twist 
map in our proof of a variant of the Ado theorem in \S \ref{sec-ado}.

\begin{definition}
A \emph{subalgebra}, respectively \emph{ideal}, of the Hom-algebra 
$(A,\cdot,\alpha)$ is a subalgebra, respectively ideal of the algebra 
$(A,\cdot)$ which is simultaneously a Hom-subspace of $(A,\alpha)$.
\end{definition}

The quotient of a Hom-vector space by a Hom-subspace, the quotient of a 
Hom-algebra by an ideal, and extension of Hom-algebras are defined in the 
standard way.

Below, when speaking about the algebra or the ideal generated by a given set, we
will mean generation in the Hom-algebra sense, i.e. the minimal algebra or ideal
containing the given set and closed under multiplication \emph{and} the twist 
map.

By abuse of notation, and as it is customary in the case of ordinary (i.e., not
Hom) algebras, if the algebra multiplication and the twist map are not needed to
be specified explicitly, or are clear from the context, we will omit them from 
the Hom-algebra notation, and denote the Hom-algebra just by the underlying 
vector space, $A$.

\begin{definition}
A Hom-algebra $(L,\liebrack,\alpha)$ is called a \emph{Hom-Lie algebra}, if $L$
is anticommutative as an algebra, i.e.
$$
[y,x] = -[x,y] ,
$$
and the following \emph{Hom-Jacobi identity} holds:
$$
[[x,y],\alpha(z)] + [[z,x],\alpha(y)] + [[y,z],\alpha(x)] = 0
$$
for any $x,y,z \in L$.
\end{definition}

Various examples of Hom-Lie algebras can be found, among others, in 
\cite{benayadi-m}, \cite{hellstrom-et-al}, \cite{mm}, \cite{MS08}, and 
\cite{makhlouf-silv}.

\begin{definition}
A linear map $\varphi: A_1 \to A_2$ is a \emph{homomorphism} of a Hom-algebra 
$(A_1, \cdot_1, \alpha_1)$ to a Hom-algebra $(A_2, \cdot_2, \alpha_2)$, if 
$\varphi$ is a homomorphism of algebras, i.e.
$$
\varphi(x \cdot_1 y) = \varphi(x) \cdot_2 \varphi(y)
$$
for any $x,y \in A_1$, and $\varphi \circ \alpha_1 = \alpha_2 \circ \varphi$.
\end{definition}

An interesting construction, provided in \cite{Yau}, allows to build a Hom-Lie 
algebra starting from a Lie algebra and an algebra homomorphism. Namely, if
$(L,\liebrack)$ is a Lie algebra, and $\alpha$ a homomorphism of $L$, then 
$(L,\alpha \liebrack,\alpha)$ is a Hom-Lie algebra, called the \emph{Yau twist}
of $L$. Hom-Lie algebras obtained in this way are called algebras of Lie type. 
In particular, if $(L,\liebrack,\alpha)$ is a multiplicative nondegenerate 
Hom-Lie algebra, then $(L,\alpha^{-1}\liebrack)$ is a Lie algebra.

\begin{definition}
Let $(L, \liebrack, \alpha)$ be a Hom-Lie algebra, and $A$ an associative 
commutative algebra. A \emph{current Hom-Lie algebra} is a Hom-Lie algebra
$(L \otimes A, \liebrack, \widehat\alpha)$ defined on the tensor product 
$L \otimes A$, where the Lie bracket (denoted by the same symbol by abuse of 
notation) is defined by
$$
[x \otimes a, y \otimes b] = [x,y] \otimes ab
$$
for $x,y \in L$, $a,b \in A$, and the twist map is defined by
$$
\widehat{\alpha} (x \otimes a) = \alpha(x) \otimes a .
$$
\end{definition}

(Note parenthetically that this current algebra construction can be generalized
to the case where $A$ is also a Hom-algebra, and, further, to algebras over 
operads Koszul dual in some extended Hom sense, but we will not need this 
generality here; hopefully, it will be treated elsewhere).

\begin{lemma}\label{l-cur}
Let $L$ be a Hom-Lie algebra, and $A$ an associative commutative algebra. If
$L$ is multiplicative or nondegenerate, then so is the current Hom-Lie algebra
$L \otimes A$.
\end{lemma}

\begin{proof}
Obvious.
\end{proof}

\begin{definition}
Let $(L,\liebrack,\alpha)$ be a Hom-Lie algebra. A linear map $D: L \to L$ is 
called an \emph{$\alpha$-derivation of $L$}, if
\begin{equation}\label{eq-der}
D([x,y]) = [D(x),\alpha(y)] + [\alpha(x),D(y)]
\end{equation}
holds for any $x,y \in L$. 

If, additionally, $L$ is multiplicative, then, in addition to (\ref{eq-der}),
we require that
$$
D \circ \alpha = \alpha \circ D.
$$
\end{definition}

In the literature there exist several different definitions of derivation of a 
Hom-Lie algebra. Say, in \cite{makhlouf-silv} one defines $1$-coboundaries 
(i.e., in fact, derivations) as ordinary derivations of the algebra structure, 
thus ignoring the twist map $\alpha$; and in \cite{sheng} one defines the 
so-called $\alpha^k$-derivations, obtained by replacing in formula (\ref{eq-der}) $\alpha$
by its $k$th power. While these definitions have their merits in appropriate 
contexts, the definition adopted here is stipulated, like the definition of 
representation below, by the structural interpretation: for a (multiplicative, 
nondegenerate) Hom-Lie algebra $L$, a linear map $D: L \to L$ is an $\alpha$-derivation, 
if and only if the direct sum $L \oplus KD$, where the multiplication between 
elements of $L$ and $D$ is determined by the action of the former on the latter,
and $\alpha(D) = D$, is a (multiplicative, nondegenerate) Hom-Lie algebra.

Similarly, there are few slightly different definitions in the literature of 
what is a representation of a Hom-Lie algebra (see, for example, 
\cite{benayadi-m}, \cite{sheng}, \cite{sheng-x}). To see what is the ``right'' 
definition, let us employ an old principle which goes back to Eilenberg: for an
algebra $L$ in a given variety of algebras, two $L$-actions on a vector space 
$V$, from the left and from the right, are declared a birepresentation of $L$, 
if the semidirect sum $L \oplus V$, where multiplication between elements of $L$
and $V$ is determined by the actions of $L$ on $V$, and multiplication on $V$ is zero, belongs to the same 
variety.

So, let $(L,\liebrack,\alpha)$ be a Hom-Lie algebra, $\rho: L \to \End(V)$ a 
linear map providing an action of $L$ on $V$ (as in the case of ordinary Lie 
algebras, because $L$ is anticommutative, the left and right actions differ only
by sign, so we can speak on -- say, left -- representations instead of 
birepresentations), and an anticommutative multiplication on the semidirect sum
$L \oplus V$ is defined by $[x,v] = \rho(x)(v)$ for $x\in L$ and $v \in V$, and
by $[V,V] = 0$. To make from it a Hom-Lie algebra, we have to extend the twist 
map $\alpha$ from $L$ to $L \oplus V$. According to the Eilenberg principle, it is 
reasonable to assume that the image of the restriction of $\alpha$ to $V$ lies 
in $V$; let us denote this restriction by $\beta \in \End(V)$. Since the 
multiplication on $V$ is zero, the Hom-Jacobi identity gives nothing for triples
where at least two elements belong to $V$; for the triple $x,y \in L$, $v \in V$
it is equivalent to
\begin{equation*}
\rho([x,y])(\beta(v)) = 
\rho(\alpha(x))(\rho(y)(v)) - \rho(\alpha(y))(\rho(x)(v)) .
\end{equation*}

If $(L,\liebrack,\alpha)$ is multiplicative, then $\alpha$, extended to the 
semidirect sum $L \oplus V$, is a homomorphism of the algebra structure if and 
only if
$$
\rho(\alpha(x))(\beta(v)) = \beta(\rho(x)(v))
$$
for any $x\in L$ and $v\in V$.

Thus we arrive at

\begin{definition}\label{def-repr}
Let $(L,\liebrack,\alpha)$ be a Hom-Lie algebra, and $(V,\beta)$ a Hom-vector 
space. A linear map $\rho: L \to \End(V)$ is called a \emph{representation of 
$L$ in $V$}, if
\begin{equation}\label{eq-repr}
\rho([x,y]) \circ \beta = 
\rho(\alpha(x)) \circ \rho(y) - \rho(\alpha(y)) \circ \rho(x)
\end{equation}
holds for any $x,y\in L$.

The representation $\rho$ is called \emph{multiplicative} if, in addition to 
(\ref{eq-repr}), 
\begin{equation}\label{eq-m}
\rho(\alpha(x)) \circ \beta = \beta \circ \rho(x)
\end{equation}
holds for any $x\in L$.

The representation is called \emph{nondegenerate} if the Hom-vector space 
$(V,\beta)$ is nondegenerate.

In all these situations, $V$ is called a (multiplicative, nondegenerate) 
\emph{module} over $L$.
\end{definition}

Note that though the main utility of multiplicative representation is to serve
as the ``right'' notion of representation in the category of multiplicative 
Hom-Lie algebras, nothing prevents us to consider multiplicative representations
of not necessarily multiplicative Hom-Lie algebras (see, for example, 
Proposition~\ref{f-tens} for the situation in which this may make sense).

The notions of the kernel of a representation, the direct sum of 
representations, and subrepresentation and submodule are defined in the obvious 
standard ways. Note that, in general, the kernel of a representation is not a 
subalgebra, but the kernel of a multiplicative and nondegenerate representation
is a subalgebra.

The center is defined exactly as in the case of Lie algebras, ignoring the 
twist map:

\begin{definition}
The \emph{center} of a Hom-Lie algebra $L$, denoted as $Z(L)$, is 
$\set{z \in L}{[z,L] = 0}$.
\end{definition}

This definition is justified by

\begin{lemma}\label{lemma-ad}
Let $L$ be a (multiplicative) Hom-Lie algebra. The map $\ad: L \to \End(L)$, 
where $\ad(x)(y) = [x,y]$, is a (multiplicative) representation of $L$, called
the \emph{adjoint representation}. The kernel of the adjoint representation 
coincides with $Z(L)$.
\end{lemma}

\begin{proof}
For $\rho = \ad$, the condition (\ref{eq-repr}) is equivalent to the Hom-Jacobi
identity, and the condition (\ref{eq-m}) is equivalent to the condition that the
twist map is the algebra homomorphism. The statement about the kernel is 
obvious.
\end{proof}

\begin{definition}
A representation $\rho$ of a Hom-Lie algebra $L$ is called \emph{nilpotent}, if
the associative algebra of linear transformations generated by $\rho(L)$ in 
$\End(L)$ is nilpotent. That is to say, there exists $n \in \mathbb N$ such that
$\rho(x_1) \circ \dots \circ \rho(x_n) = 0$ for any $x_1, \dots, x_n \in L$. The
minimal such $n$ is called the \emph{index of nilpotency}, or \emph{nilindex}, 
of $\rho$.
\end{definition}

\begin{definition}\label{def-tens}
Let $L$ be a Hom-Lie algebra, and $\rho: L \to \End(V)$, $\tau: L \to \End(W)$ 
be two representations of $L$ in the Hom-vector spaces $(V,\beta)$ and 
$(W,\gamma)$ respectively. Then
the map 
$\rho \otimes \tau: L \to \End(V \otimes W)$ defined as
$$
(\rho \otimes \tau)(x) = \rho(x) \otimes \gamma + \beta \otimes \tau(x)
$$
for $x \in L$, is a called the \emph{tensor product} of representations $\rho$
and $\tau$.
\end{definition}

\begin{proposition}\label{f-tens}
The tensor product of two multiplicative representations of a Hom-Lie algebra 
$L$ in Hom-vector spaces $(V,\beta)$ and $(W,\gamma)$, is a multiplicative 
representation of $L$ in the Hom-vector space $(V \otimes W, \beta \otimes \gamma)$.
\end{proposition}

\begin{proof}
To check (\ref{eq-repr}): for any $x,y \in L$ we have
\begin{align*}
(\rho \otimes \tau)([x,y]) &\circ (\beta \otimes \gamma)
\\ &\overset{\bf (1)}= 
\big(\rho([x,y]) \circ \beta\big) \otimes \gamma^2 + 
\beta^2 \otimes \big(\tau([x,y]) \circ \gamma\big) 
\\ &\overset{\bf (2)}=
\big(\rho(\alpha(x)) \circ \rho(y) - \rho(\alpha(y)) \circ \rho(x)\big) \otimes
\gamma^2
+
\beta^2 \otimes 
\big(\tau(\alpha(x)) \circ \tau(y) - \tau(\alpha(y)) \circ \rho(x)\big)
\\ &+
\Big(
 \big(\beta \circ \rho(y)\big) \otimes \big(\tau(\alpha(x)) \circ \gamma\big)
- \big(\rho(\alpha(y)) \circ \beta\big) \otimes \big(\gamma \circ \tau(x)\big)
\Big)
\\ &+
\Big(
 \big(\rho(\alpha(x)) \circ \beta \big) \otimes \big(\gamma \circ \tau(y) \big)
- \big(\beta \circ \rho(x) \big) \otimes \big(\tau(\alpha(y)) \circ \gamma \big)
\Big)
\\ &\overset{\bf (3)}=
\Big(
 \big(\rho(\alpha(x)) \circ \rho(y) \otimes \gamma^2
+ \big(\beta \circ \rho(y)\big) \otimes \big(\tau(\alpha(x)) \circ \gamma\big)
\\ &+ 
\big(\rho(\alpha(x)) \circ \beta \big) \otimes \big(\gamma \circ \tau(y) \big)
+ \beta^2 \otimes \tau(\alpha(x)) \circ \tau(y) 
\Big)
\\ &-
\Big(
 \big(\rho(\alpha(y)) \circ \rho(x)\big) \otimes \gamma^2
+ \big(\beta \circ \rho(x) \big) \otimes \big(\tau(\alpha(y)) \circ \gamma \big)
\\ &+ 
\big(\rho(\alpha(y)) \circ \beta\big) \otimes \big(\gamma \circ \tau(x)\big)
+ \beta^2 \otimes \tau(\alpha(y)) \circ \rho(x)
\Big)
\\ &\overset{\bf (4)}=
(\rho \otimes \tau)(\alpha(x)) \circ \Big(\rho(y) \otimes \gamma + \beta \otimes \tau(y)\Big)
-
(\rho \otimes \tau)(\alpha(y)) \circ \Big(\rho(x) \otimes \gamma + \beta \otimes \tau(x)\Big)
\\ &\overset{\bf (5)}=
(\rho \otimes \tau)(\alpha(x)) \circ (\rho \otimes \tau)(y)
-
(\rho \otimes \tau)(\alpha(y)) \circ (\rho \otimes \tau)(x) .
\end{align*}
Here in the equality \textbf{(1)} we use the definition of $\rho \otimes \tau$,
in \textbf{(2)} we use the fact that $\rho$ and $\tau$ are representations 
(formula (\ref{eq-repr})), and add terms which cancel out due to (\ref{eq-m}), 
in \textbf{(3)} we rearrange terms, and in \textbf{(4)} and \textbf{(5)} we use 
again repeatedly the definition of $\rho \otimes \tau$.

To check (\ref{eq-m}): for any $x \in L$ we have
\begin{multline*}
(\rho \otimes \tau)(\alpha(x)) \circ (\beta \otimes \gamma) 
\overset{\bf (1)}= 
 \big(\rho(\alpha(x)) \circ \beta\big) \otimes \gamma^2
+ \beta^2 \otimes \big(\tau(\alpha(x)) \circ \gamma\big) 
\\ \overset{\bf (2)}=
 \big(\beta \circ \rho(x)\big) \otimes \gamma^2
+ \beta^2 \otimes \big(\gamma \circ \tau(x)\big) 
=
(\beta \otimes \gamma) \circ 
\big(\rho(x) \otimes \gamma + \beta \otimes \tau(x)\big)
\overset{\bf (3)}=
(\beta \otimes \gamma) \circ (\rho \otimes \tau)(x).
\end{multline*}
Here in the equalities \textbf{(1)} and \textbf{(3)} we use the definition of 
$\rho \otimes \tau$, and in \textbf{(2)} we use twice the formula (\ref{eq-m}).
\end{proof}

Note that it is no longer true that the tensor product of not necessarily 
multiplicative representations, even of a multiplicative Hom-Lie algebra, is
a representation.

\begin{lemma}\label{lemma-tp}
The tensor product of two nilpotent multiplicative representations of a Hom-Lie 
algebra is nilpotent.
\end{lemma}

\begin{proof}
Let $\rho$, $\tau$ be two nilpotent multiplicative representations of a Hom-Lie
algebra $L$, with indices of nilpotency $n$ and $m$, respectively. For each 
$x_1, \dots, x_{n+m-1} \in L$, the expression 
\begin{equation}\label{eq-rt}
(\rho \otimes \tau)(x_1) \circ \dots \circ (\rho \otimes \tau)(x_{n+m-1})
\end{equation}
is equal to the sum of elements of the form $B \otimes C$, where the first 
tensor factor $B$ is of the form 
\begin{equation}\label{eq-ind}
\beta^{r_1} \circ \rho(x_{i_1}) \circ \beta^{r_2} \circ \rho(x_{i_2}) \circ 
\dots \circ \beta^{r_k} \circ \rho(x_{i_k}) \circ \beta^{r_{k+1}}
\end{equation}
for some (possibly empty) subset of indices 
$\{i_1, \dots, i_k\} \subseteq \{1, \dots, n+m-1\}$, and some (possibly zero) 
$r_1, \dots, r_{k+1} \in \mathbb N$, the second tensor factor $C$ has the same 
form with $\beta$ replaced by $\gamma$, and $\rho$ by $\tau$, and the respective
subsets of indices occurring in $B$ and $C$ are complementary in 
$\{1,\dots,n+m-1\}$ (so the sum of the number of $\rho(x_i)$'s in the first
tensor factor and in the second tensor factor is equal to $n+m-1$).

Using the relation (\ref{eq-m}), one can put in the expression (\ref{eq-ind})
all $\beta$'s ``to the right'', getting ``at the left'' expression of the form 
\begin{equation}\label{eq-ex1}
\rho(\alpha^{s_1}(x_{i_1})) \circ \rho(\alpha^{s_2}(x_{i_2})) \circ \dots \circ 
\rho(\alpha^{s_k}(x_{i_k}))
\end{equation}
for some (possibly zero) $s_1, \dots, s_k \in \mathbb N$. Similarly, in the 
second tensor factor $C$ we can get ``at the left'' expression of the form
\begin{equation}\label{eq-ex2}
\tau(\alpha^{t_1}(x_{j_1})) \circ \tau(\alpha^{t_2}(x_{j_2})) \circ \dots \circ 
\tau(\alpha^{t_\ell}(x_{j_\ell}))
\end{equation}
for some $t_1, \ldots, t_\ell \in \mathbb N$.

Since $k+\ell = n+m-1$, we have either $k \ge n$, or $\ell \ge m$, and at least
one of expressions (\ref{eq-ex1}) and (\ref{eq-ex2}) vanishes, thus the whole 
expression (\ref{eq-rt}) vanishes, which shows that $\rho \otimes \tau$ is 
nilpotent of index $\le n+m-1$.
\end{proof}

\section{Hom-associative algebras and their representations}\label{sec-ass}

In this section we recall basic definitions and facts related to Hom-associative
algebras (see \cite{MS08}, \cite{Yau}), similarly how it was done for Hom-Lie 
algebras in the previous section. Though Hom-as\-so\-ci\-a\-ti\-ve algebras are
not our primary concern in this paper, it is necessary to consider them in relation to 
Hom-Lie algebras. Our exposition culminates in Theorem \ref{th-a} which establishes an equivalence of
the two statements which can be considered as 
``Ado theorem for Hom-Lie algebras'' -- existence of a faithful representation,
and embedding into a Hom-associative algebra -- similarly with the ordinary Lie
case.

\begin{definition}\label{def-homass}
A Hom-algebra $(A,\cdot,\alpha)$ is called a \emph{Hom-associative algebra}, if
the following \emph{Hom-as\-so\-ci\-a\-ti\-ve identity} holds:
$$
(x \cdot y) \cdot \alpha(z) = \alpha(x) \cdot (y \cdot z)
$$
for any $x,y,z \in A$.

An element $u\in A$ is called a \emph{unit} if $x\cdot u=u\cdot x=\alpha (x)$ 
for any $x\in A$.
\end{definition}

\begin{definition}
If $(A,\cdot,\alpha)$ is a Hom-algebra, then the Hom-algebra defined on the same
Hom-vector space $(A,\alpha)$, and with multiplication $\liebrack$ defined by 
the commutator: 
$$
[x,y] = x \cdot y - y \cdot x
$$
for $x,y\in A$, is called an algebra \emph{associated to $A$}, and is denoted by
$(A^{(-)}, \liebrack, \alpha)$.
\end{definition}

\begin{proposition}
A Hom-algebra associated to a Hom-associative algebra $A$, is Hom-Lie. If $A$
is multiplicative (respectively, nondegenerate), then the associated Hom-algebra
$A^{(-)}$ is multiplicative (respectively, nondegenerate).
\end{proposition}

\begin{proof}
An easy computation, similar to its (classical) non-Hom counterpart. The 
statement about multiplicativity and nondegeneracy is evident.
\end{proof}

Performing the same Eilenberg game as in \S \ref{sec-p} in the class of 
Hom-associative algebras, we arrive at

\begin{definition}\label{def-ass}
Let $(A,\cdot,\alpha)$ be a Hom-associative algebra, and $(V,\beta)$ a 
Hom-vector space. 

A linear map $\rho_L: A \to \End(V)$ is called a 
\emph{left representation of $A$ in $V$}, if
\begin{equation}\label{eq-l}
\rho_L(x \cdot y) \circ \beta = \rho_L(\alpha(x)) \circ \rho_L(y)
\end{equation}
holds for any $x,y \in A$. 

A linear map $\rho_R: A \to \End(V)$ is called a 
\emph{right representation of $A$ in $V$}, if
\begin{equation}\label{eq-r}
\beta \circ \rho_R(x \cdot y) = \rho_R(x) \circ \rho_R(\alpha(y)) 
\end{equation}
holds for any $x,y \in A$. 

A pair of linear maps $(\rho_L, \rho_R)$, where $\rho_L, \rho_R: A \to \End(V)$,
is called a \emph{birepresentation of $A$ in $V$}, if $\rho_L$ is a left 
representation, $\rho_R$ is a right representation, and, additionally, the following 
compatibility condition
\begin{equation}\label{eq-comp}
\rho_L(\alpha(x)) \circ \rho_R(y) = \rho_L(x) \circ \rho_R(\alpha(y))
\end{equation}
holds for any $x,y \in A$

The left representation $\rho_L$ is called \emph{multiplicative} if, in 
addition to (\ref{eq-l}),
\begin{equation*}
\rho_L(\alpha(x)) \circ \beta = \beta \circ \rho_L(x)
\end{equation*}
holds for any $x,y \in A$; the right representation $\rho_R$ is called 
\emph{multiplicative} if, in addition to (\ref{eq-r}),
$$
\rho_R(\alpha(x)) \circ \beta = \beta \circ \rho_R(x)
$$
holds for any $x,y\in A$; and the birepresentation $\rho_L,\rho_R$ is called 
\emph{multiplicative}, if $\rho_L$ is a left multiplicative representation, and
$\rho_R$ is a right multiplicative representation.

The representation or birepresentation is called \emph{nondegenerate} if the 
Hom-vector space $(V,\beta)$ is nondegenerate.

In all these situations, $V$ is called a (left, right, bi, multiplicative, 
nondegenerate) \emph{module} over $A$. 
\end{definition}

This definition extends those adopted in \cite{makhlouf-silv} (where left 
representations are considered), and is a particular case of those in 
\cite{graziani-et-al} (where representations of BiHom-algebras -- in which the
associative identity is deformed by two twist maps instead of one -- are 
considered).

Note that as in the ordinary associative case, we cannot define in any 
reasonable way the tensor product of two representations of a Hom-associative 
algebra, multiplicative or not.

The following is a Hom-associative analog of Lemma \ref{lemma-ad}.

\begin{lemma}\label{f-adj}
Let $A$ be a Hom-associative algebra. Then the map sending an element of $A$
to the left (respectively, right) multiplication on that element, is a left 
(respectively, right) representation of $A$ in itself, and together they form
a birepresentation of $A$ in itself. The kernel of this left (respectively, 
right) representation coincide with the left annulator
$$
\Ann_L(A) = \set{x \in A}{xA = 0} ,
$$
and with the right annulator
$$
\Ann_R(A) = \set{x \in A}{Ax = 0} ,
$$
respectively.
\end{lemma}

\begin{proof}
All the three conditions defining a birepresentation -- of left representation
(\ref{eq-l}), of right representation (\ref{eq-r}), and the compatibility 
condition (\ref{eq-comp}) -- are equivalent to the Hom-associative identity.
The statement about the kernels is obvious.
\end{proof}

\begin{lemma}
Let $A$ be an associative algebra with unit, and $a \in A$ an invertible 
element. Then $(A, \cdot, \Ad_a)$, where the new multiplication $\cdot$ in $A$ 
is defined by 
$$
x \cdot y = axa^{-1}ya^{-1}
$$
for $x,y \in A$, and $\Ad_a: A \to A$ is the adjoint (in the group sense) map: 
$$
\Ad_a(x) = axa^{-1}
$$
for $x \in A$, is a multiplicative nondegenerate Hom-associative algebra.
\end{lemma}

\begin{proof}
Straightforward computation, see \cite[Proposition 4.1]{sheng-x}, or 
\cite[Lemma 4.3]{graziani-et-al} for a more general BiHom-associative case. The
nondegeneracy is evident.
\end{proof}

\begin{lemma}\label{f-ass-end}
Let $A$ be a Hom-associative algebra, and $(V,\beta)$ a Hom-vector space with 
$\beta$ invertible. A linear map $\rho: A \to \End(V)$ is a left (respectively,
right) multiplicative representation of $A$, if and only if $\rho$ is a 
homomorphism from the Hom-associative algebra $A$ to the Hom-associative algebra
$(\End(V), \circ, \Ad_\beta)$ (respectively, to the Hom-associative algebra 
$(\End(V), \circ, \Ad_{\beta^{-1}})$).
\end{lemma}

\begin{proof}
Straightforward computation, see \cite[Theorem 4.2]{sheng-x} for a similar
computation in the Hom-Lie case, or \cite[Proposition 4.4]{graziani-et-al} for
a more general BiHom-associative case.
\end{proof}

The following could be dubbed as ``Ado theorem for Hom-associative algebras''. 
Of course, its associative analog is trivial (adjoin a unit, and consider the 
representation in itself), and the proof below just mimics this triviality in 
the Hom case.

\begin{lemma}\label{f-5}
Any (finite-dimensional) multiplicative nondegenerate Hom-associative algebra 
admits a faithful (finite-di\-men\-si\-o\-nal) left (or right) nondegenerate 
representation.
\end{lemma}

\begin{proof}
Let $(A,\cdot,\alpha)$ be a multiplicative nondegenerate Hom-associative 
algebra. Adjoin to $A$ a unit (in the Hom sense, see 
Definition~\ref{def-homass}); namely, consider the direct sum 
$\widehat A = A \oplus K 1_\alpha$, and extend to it the multiplication and the
twist map as follows:
$$
x \cdot 1_\alpha = 1_\alpha \cdot x = \alpha(x)
$$
for any $x \in \widehat A$, and
\begin{gather*}
\alpha(1_\alpha) = 1_\alpha .
\end{gather*}

Then $(\widehat A, \cdot, \alpha)$ is a multiplicative nondegenerate 
Hom-associative algebra, containing $A$ as a subalgebra. By Lemma \ref{f-adj}, 
the left (right) representation of $\widehat A$ in itself has zero kernel, and 
its restriction to $A$ provides the desired representation.
\end{proof}

\begin{lemma}\label{f-4}
Let $L$ be a Hom-Lie algebra, and $(V, \beta)$ a Hom-vector space with $\beta$
invertible. A linear map $\rho: L \to \End(V)$ is a multiplicative 
representation of $L$, if and only if $\rho$ is a homomorphism from the Hom-Lie
algebra $L$ to the Hom-Lie algebra $(\End(V)^{(-)}, \liebrack, \Ad_\beta)$ 
(where $\liebrack$ is a commutator with respect to composition of linear maps on $V$).
\end{lemma}

\begin{proof}
This is Theorem 4.2 from \cite{sheng-x}, or a particular case of 
Proposition 4.10 from \cite{graziani-et-al}.
\end{proof}

\begin{theorem}\label{th-a}
For a Hom-Lie algebra $L$, the following are equivalent:
\begin{enumerate}[\upshape(i)]
\item
$L$ admits a finite-dimensional faithful multiplicative nondegenerate 
representation;
\item
$L$ is embedded into a Hom-Lie algebra of the form $A^{(-)}$, where $A$ is a
finite-dimensional multiplicative nondegenerate Hom-associative algebra.
\end{enumerate}
\end{theorem}

\begin{proof}
Let us introduce another auxiliary condition:
\begin{enumerate}[\upshape(i)]
\setcounter{enumi}{2}
\item
$L$ is embedded into a Hom-Lie algebra of the form 
$(\End(V)^{(-)}, \liebrack, \Ad_\beta)$ for some fi\-ni\-te-di\-men\-si\-o\-nal 
nondegenerate Hom-vector space $(V,\beta)$.
\end{enumerate}

Then (iii) $\Rightarrow$ (ii) is obvious, and (i) $\Leftrightarrow$ (iii) 
follows from Lemma \ref{f-4}.

(ii) $\Rightarrow$ (iii): By Lemma \ref{f-5}, $A$ admits a faithful (say, left) 
representation in a finite-dimensional nondegenerate Hom-vector space 
$(V,\beta)$, and then by Lemma \ref{f-ass-end}, $A$ is embedded into a 
Hom-associative algebra of the form $(\End(V), \circ, \Ad_\beta)$. Consequently,
$A^{(-)}$, and hence $L$, is embedded into the Hom-Lie algebra 
$(\End(V)^{(-)}, \liebrack, \Ad_\beta)$.
\end{proof}

\section{Nilpotent algebras}\label{sec-nilp}

According to \cite{mm}, the notion of nilpotency is carried over to the Hom-Lie
case verbatim: 

\begin{definition}
A Hom-Lie algebra $(L,\liebrack,\alpha)$ is called \emph{nilpotent} if it is 
nilpotent as an algebra; that is, $L^n = 0$ for some $n \in \mathbb N$, where 
$L^n$, the members of the lower central series of $L$, are defined inductively: 
$L^1 = L$, and $L^n = [L^{n-1},L]$ for $n>1$. The minimal number $n$ such that
$L^n = 0$ is called \emph{the degree of nilpotency}, or \emph{nilindex}. A 
nilpotent Hom-Lie algebra of degree $2$ (i.e., with zero multiplication) is 
called \emph{abelian}.
\end{definition}

Note that for a multiplicative Hom-Lie algebra $L$ we have 
$\alpha(L^n) \subseteq L^n$ for any $n$.

Note also that the Yau twist of a nilpotent Lie algebra $(L,[\cdot ,\cdot])$ 
along a Lie algebra homomorphism $\alpha: L \rightarrow L$, is a nilpotent 
Hom-Lie algebra; that is, $(L,[\cdot ,\cdot ]_{\alpha},\alpha)$ is a nilpotent 
Hom-Lie algebra, where $\left[x,y\right]_{\alpha}=[\alpha(x),\alpha(y)]$. 
Conversely, if $(L,[\cdot ,\cdot ],\alpha)$ is a multiplicative nondegenerate 
nilpotent Hom-Lie algebra, then its untwist $(L,\alpha^{-1}[\cdot ,\cdot ])$ is
a nilpotent Lie algebra.

\begin{definition}\label{def-st}
A Hom-Lie algebra $L$ is called \emph{strongly nilpotent}, if there is a chain
of ideals of $L$, starting with $0$ and ending with $L$:
\begin{equation}\label{eq-i}
0 = I_n \vartriangleleft I_{n-1} \vartriangleleft \dots \vartriangleleft I_2 
\vartriangleleft I_1 = L
\end{equation}
such that for each $1 < i < n$, $\dim I_i/I_{i+1} = 1$ and 
$[L,I_i] \subseteq I_{i+1}$.
\end{definition}



\begin{lemma}\label{lemma-ij}
A finite-dimensional Hom-Lie algebra $L$ is strongly nilpotent if and only if it
satisfies the following property: every nonzero ideal $I$ of $L$ contains an 
ideal $J$ of codimension $1$ in $I$ such that $[L,I] \subseteq J$.
\end{lemma}

\begin{proof}
The ``if'' part: starting with $I_1=L$, we build inductively the chain of ideals
such that $I_{i+1}$ is of codimension $1$ in $I_i$, and 
$[L,I_i] \subseteq I_{i+1}$. Since $L$ is finite-dimensional, this chain will 
terminate at $0$.

The proof of the ``only if'' part repeats verbatim the proof of 
\cite[Lemma 2.8]{ado}, with ``Lie'' being replaced by ``Hom-Lie'', and 
``nilpotent'' by ``strongly nilpotent''.
\end{proof}

\begin{lemma}\label{lemma-nilp}
A strongly nilpotent Hom-Lie algebra is nilpotent.
\end{lemma}

\begin{proof}
The proof is exactly the same as in the Lie case (see, for example, 
\cite[\S 4, Proposition 1]{bourbaki}). Let $L$ be a strongly nilpotent Hom-Lie
algebra with the chain of ideals (\ref{eq-i}). Then by induction we have 
$L^i \subseteq I_i$ for any $1 \le i \le n$; in particular, 
$L^n \subseteq I_n = 0$.
\end{proof}

It is well known that in the class of Lie algebras the converse is true, i.e., 
the notions of nilpotency and strong nilpotency coincide. This is not so in the
class of Hom-Lie algebras: a trivial example is provided by an abelian Hom-Lie 
algebra $L$ of dimension $>1$. Ideals in $L$ are exactly Hom-subspaces, so if we
choose the twist map $\alpha$ in such a way that there exists no proper 
$\alpha$-invariant subspaces, $L$ will be not strongly nilpotent. Of course, 
this requires the ground field to be not algebraically closed. However, over 
algebraically closed fields, and in the class of finite-dimensional 
multiplicative Hom-Lie algebras, the notions of nilpotency and strong 
nilpotency, like in the Lie case, do coincide. To prove this, we need the 
following elementary linear-algebraic fact.

\begin{lemma}\label{l-ab}
For any two Hom-subspaces $A \subsetneqq B$ of a finite-dimensional Hom-vector 
space over an algebraically closed field, there is a third Hom-subspace $C$ such
that $A \subseteq C \subset B$, and $C$ is of codimension $1$ in $B$.
\end{lemma}

(In essence, this says that a finite-dimensional abelian Hom-Lie algebra over an
algebraically closed field is strongly nilpotent).

\begin{proof}
In fact, for any integer $n$ such that $\dim A \le n \le \dim B$, there exists a
Hom-subspace $C$ of dimension $n$ sitting between $A$ and $B$. This is proved by
induction on $n$: the induction step consists of taking the quotient $B/C$ by 
the $n$-dimensional Hom-subspace $C$ containing $A$, taking eigenvector (i.e., 
an $1$-dimensional Hom-subspace) of the twist map in this quotient, and passing
back to its preimage in $B$.
\end{proof}

\begin{proposition}\label{prop-nilp}
A finite-dimensional multiplicative Hom-Lie algebra over an algebraically closed
field is nilpotent if and only if it is strongly nilpotent.
\end{proposition}

\begin{proof}
The ``if'' part is covered by Lemma \ref{lemma-nilp}, so let us prove the 
``only if'' part. 

Let $I$ be a nonzero ideal in a finite-dimensional nilpotent multiplicative 
Hom-Lie algebra $L$ defined over an algebraically closed field. First note that
$[L,I] \ne I$. Indeed, if $[L,I] = I$, then for any $n\in \mathbb N$, 
$[L, \dots, [L,[L,I]] \dots ] = I$ (where $L$ occurs $n$ times). But the 
left-hand side in the last equality lies in $L^{n+1}$, and hence vanishes for 
some $n$, a contradiction.

Since $L$ is multiplicative, 
$\alpha([L,I]) \subseteq [\alpha(L),\alpha(I)] \subseteq [L,I]$, and hence 
$[L,I]$ is an ideal in $L$. By Lemma \ref{l-ab}, there is a Hom-subspace $J$ of
$L$ such that $[L,I] \subseteq J \subset I$, and $J$ is of codimension $1$ in 
$I$. Then $[L,J] \subseteq [L,I] \subseteq J$, so $J$ is an ideal.

Thus we see that any nonzero ideal $I$ of $L$ contains an ideal $J$ of 
codimension $1$ in $I$ such that $[L,I] \subseteq J$, and by 
Lemma~\ref{lemma-ij} $L$ is strongly nilpotent.
\end{proof}

\section{Free algebras}\label{sec-free}

In this section we review the construction and elementary properties of free and
free nilpotent Hom-Lie algebras. As the previous sections, it does not contain 
anything really new; everything here is either implicitly available in the 
literature (see, for example, \cite[\S 49.2]{razmyslov} which treats the more 
general case of free linear algebras with several multiary operations), or is a
straightforward application of the general universal-algebraic machinery.

The main outcome of this section is the fact -- sounding almost tautologically 
trivial -- that any finite-dimensional nilpotent (multiplicative, nondegenerate)
Hom-Lie algebra is a homomorphic image of a -- suitably defined -- 
finite-dimensional $\mathbb N$-graded free nilpotent (multiplicative, 
nondegenerate) Hom-Lie algebra.

We start with an obvious

\begin{definition}
Let $G$ be an abelian semigroup. A Hom-Lie algebra $(L,\liebrack,\alpha)$ is 
called \emph{$G$-graded} if it is $G$-graded as an algebra, i.e. $L$ is 
decomposed as the direct sum $L = \bigoplus_{g \in G} L_g$ with 
$[L_g,L_h] \subseteq L_{g+h}$ for any $g,h \in G$, and each homogeneous 
component is stable under $\alpha$: $\alpha(L_g) \subseteq L_g$ for any 
$g\in G$.
\end{definition}

By abstract universal algebraic nonsense, free Hom-Lie algebras can be defined
as just free algebraic systems in the variety of Hom-Lie algebras considered as
vector spaces with one binary and one unary operation. 

A more concrete construction can run as follows. Let $X$ be a set. Consider all
(nonassociative) words in $X$, and all possible ``applications'' of $\alpha$ on
subwords in those words. More precisely, define recursively 
\begin{equation}\label{eq-f1}
F_1(X) = \set{\alpha^\ell(x)}{x \in X, \, \ell \in \mathbb N}
\end{equation}
(we assume that $\alpha^0(x) = x$), and, for $n > 1$, 
\begin{equation}\label{eq-fn}
F_n(X) = \bigcup_{k=1}^{n-1} \>
\set{\alpha^\ell(uv)}{u \in F_k(X), v\in F_{n-k}(X), \, \ell\in \mathbb N} .
\end{equation}

The multiplication between elements of $F(X) = \cup_{n \ge 1} F_n(X)$ is 
performed by concatenation; thus, $F_n(X) F_m(X) \subseteq F_{n+m}(X)$. The 
twist map $\alpha$ maps a word $u \in F(X)$ to the word $\alpha(u)$, and 
applying the relation $\alpha(\alpha^\ell(v)) = \alpha^{\ell+1}(v)$ for any 
$\ell \in \mathbb N$ and $v \in F(X)$; thus, $\alpha(F_n(X)) \subseteq F_n(X)$. 
Then $F(X)$ forms a magma with a unary operation $\alpha$, and the magma 
algebra $KF(X)$ over the ground field $K$, with $\alpha$ extended by linearity,
forms the free Hom-algebra with the generating set $X$. It is obvious that 
$KF(X)$ is an $\mathbb N$-graded Hom-algebra: 
\begin{equation}\label{eq-gr}
KF(X) = \bigoplus_{n\in \mathbb N} KF_n(X) .
\end{equation}

The free Hom-Lie algebra $\mathcal L(X)$ freely generated by a set $X$ is 
obtained as a quotient of $KF(X)$ by the ideal generated by elements of the form
$uv + vu$ and $(uv)\alpha(w) + (wu)\alpha(v) + (vw)\alpha(u)$, where 
$u,v,w \in KF(X)$. Since this ideal is homogeneous with respect to the grading 
(\ref{eq-gr}), $\mathcal L(X)$ is $\mathbb N$-graded too. 

Of course, this is equivalent to constructing $\mathcal L(X)$ directly by the 
inductive process (\ref{eq-f1})--(\ref{eq-fn}), where instead of juxtaposition 
$uv$ one takes the bracket $[u,v]$ (which is assumed to be anticommutative and 
satisfying the Hom-Jacobi identity). Clearly, the inductively constructed 
elements in this case will be no longer linearly independent, and the question 
of constructing a basis of a free Hom-Lie algebra, similar to one of the known 
bases of free Lie algebras, seems to be a difficult one.

In \cite[\S 3]{hellstrom-et-al} the elements of free Hom-algebras (Lie and 
associative) are represented as labeled binary trees: branching, as usual, 
corresponds to the binary multiplication, and labels are equal to the exponent 
$\ell$ in ``application'' of the power $\alpha^\ell$ to the given vertex.

According to the general universal-algebraic principles, it would be natural to
define the free nilpotent Hom-Lie algebra $\mathcal N_n(X)$ of degree $n$ and 
freely generated by a set $X$, as the quotient of $\mathcal L(X)$ by the ideal 
generated by $\mathcal L(X)^n$. This definition is, however, unsatisfactory for
our purposes: below, when proving an analog of the Ado theorem for nilpotent 
Hom-Lie algebras, we want to stay in the category of finite-dimensional 
algebras, but $\mathcal N_n(X)$, unlike its ordinary Lie-algebraic counterpart,
is obviously infinite-dimensional, as for every its element $u$ it contains all
the powers $\alpha^\ell(u)$, all of them are linearly independent. This can be 
remedied in the following way. Observe that for a finite-dimensional Hom-Lie 
algebra, the twist map $\alpha$, being a linear map on a finite-dimensional vector 
space, satisfies some polynomial equation 
\begin{equation}\label{eq-a}
f(\alpha) = 0 .
\end{equation} 

We fix this equation and add it -- or, rather, the identity $f(\alpha)(x) = 0$ 
for any $x\in L$ -- to the defining relations of the corresponding free algebra,
and define $\mathcal N_{n,f}(X)$ as the quotient of $\mathcal L(X)$ by the ideal
generated by $\mathcal L(X)^n$ and $f(\alpha)(\mathcal L(X))$ (or, what is the 
same, as the quotient of $\mathcal N_n(X)$ by the ideal generated by 
$f(\alpha)(\mathcal N_n(X))$). The ensuing algebra $\mathcal N_{n,f}(X)$ is, 
obviously, finite-dimensional. The ideal $\mathcal L(X)^n$ of $\mathcal L(X)$ is
homogeneous, and, since $\alpha$ preserves the $\mathbb N$-grading of 
$\mathcal L(X)$, the ideal generated by $f(\alpha)(\mathcal L(X))$ is 
homogeneous too, so $\mathcal N_{n,f}(X)$ remains to be $\mathbb N$-graded.
Any finite-dimensional nilpotent Hom-Lie algebra with a generating set $X$, of 
degree $n$, and satisfying the condition (\ref{eq-a}), is a quotient of $\mathcal N_{n,f}(X)$. 

In the multiplicative variant of all these, we need additionally factorize by
the ideal generated by elements of the form 
\begin{equation}\label{eq-alpha}
\alpha([u,v]) - [\alpha(u),\alpha(v)] ,
\end{equation}
where $u,v$ are elements from the respective free Hom-algebra. Roughly speaking,
that means that we may move all $\alpha$'s to the ``innermost positions''. In 
terms of the inductive process (\ref{eq-f1})--(\ref{eq-fn}) that means that 
keeping (\ref{eq-f1}), we may define the free multiplicative Hom-algebra as a 
free (nonassociative) algebra freely generated by the set $F_1(X)$. In terms of
the labeled trees used in \cite{hellstrom-et-al} that means that we label only
the terminal vertices. Let us denote the multiplicative analog of 
$\mathcal N_{n,f}(X)$, i.e. the quotient of the latter algebra by the ideal 
generated by the elements of the form (\ref{eq-alpha}), by 
$\mathcal M_{n,f}(X)$. Obviously, $\mathcal M_{n,f}(X)$ remains to be 
$\mathbb N$-graded. Any finite-dimensional nilpotent multiplicative Hom-Lie 
algebra with a generating set $X$, of degree $n$, and satisfying the condition 
(\ref{eq-a}), is a quotient of $\mathcal M_{n,f}(X)$. 

Note that while finite-dimensional nondegenerate Hom-algebras are closed with 
respect to homomorphic images, nondegenerate Hom-algebras in general are not 
closed, so they do not form a variety; but they form a quasivariety, and hence 
we can still speak about free nondegenerate Hom-Lie algebras. In a 
finite-dimensional situation the nondegeneracy of $\alpha$ can be expressed in 
terms of the polynomial $f$: indeed, we may take $f$ to be the characteristic or the minimal
polynomial of $\alpha$, and then the nondegeneracy of $\alpha$ is equivalent to
the nonvanishing of the free term of $f$. Thus, free (multiplicative) 
nondegenerate nilpotent Hom-Lie algebras are merely $\mathcal N_{n,f}(X)$ 
(or $\mathcal M_{n,f}(X)$) with $f$ having the nonvanishing free term.

It is clear that all the free Hom-algebras considered here do not depend on the
set $X$ itself, but merely on its cardinality $|X|$. In particular, instead of 
$\mathcal N_{n,f}(X)$ and $\mathcal M_{n,f}(X)$, we will write 
$\mathcal N_{k,n,f}$ and $\mathcal M_{k,n,f}$ respectively, where $k = |X|$.

\section{Ado theorem for nilpotent Hom-Lie algebras}\label{sec-ado}

As explained in the Introduction, we follow the scheme of \cite{ado}.

It is obvious that if a Hom-Lie algebra has a finite-dimensional faithful
(multiplicative, nondegenerate) representation, then so does any its subalgebra.

\begin{lemma}\label{l-e}
A finite-dimensional $\mathbb N$-graded Hom-Lie algebra $L$ is embedded into the
current Hom-Lie algebra $L \otimes tK[t]/(t^n)$ for some $n \in \mathbb N$.
\end{lemma}

(Note that a finite-dimensional $\mathbb N$-graded algebra is necessarily 
nilpotent).

\begin{proof}
Verbatim repetition of the proof of \cite[Lemma 2.1]{ado}.
\end{proof}

\begin{lemma}\label{l-d}
If a finite-dimensional (nilpotent, multiplicative, nondegenerate) Hom-Lie 
algebra has a nondegenerate $\alpha$-derivation, then it has a 
finite-dimensional faithful (nilpotent, multiplicative, nondegenerate) 
representation.
\end{lemma}

\begin{proof}
Let $L$ be a finite-dimensional Hom-Lie algebra having a nondegenerate 
$\alpha$-derivation $D$. The desired representation is the action of $L$ on the
ambient Hom-Lie algebra $L \oplus KD$. If $L$ is nilpotent, or multiplicative, 
or nondegenerate, then this action is respectively nilpotent, or multiplicative,
or nondegenerate too.
\end{proof}

Note that a classical result of Jacobson, \cite{jacobson}, says that a 
finite-dimensional Lie algebra over a field of characteristic zero having a 
nondegenerate derivation, is necessarily nilpotent. It appears to be an 
interesting question whether the same is true for Hom-Lie algebras; if true, the
condition of nilpotency in Lemma~\ref{l-d} is redundant.

One can generalize Lemma~\ref{l-d} by considering $1$-cocycles in an arbitrary 
module instead of $\alpha$-derivations, similarly to the ordinary Lie case 
(\cite[Lemma 2.3]{ado}). However, this will require us to analyze various 
definitions in the literature of cohomology of Hom-Lie algebras, even more 
diverse then those of derivations or representations; we want to avoid this 
task here, and confine ourselves with Lemma \ref{l-d} which is enough for our 
purposes.

\begin{lemma}\label{l-f}
A finite-dimensional (multiplicative) nondegenerate $\mathbb N$-graded Hom-Lie 
algebra over a field of characteristic zero has a finite-dimensional faithful 
nilpotent (multiplicative) nondegenerate representation.
\end{lemma}

\begin{proof}
Let $(L,\liebrack,\alpha)$ be such a Hom-Lie algebra. By Lemma \ref{l-e}, $L$ is
embedded into the nilpotent current Hom-Lie algebra $L \otimes tK[t]/(t^n)$. By 
Lemma~\ref{l-cur}, the latter Hom-Lie algebra is nondegenerate, and is 
multiplicative if $L$ is so. The map 
$\alpha \otimes t\frac{\dcobound}{\dcobound t}$ is a nondegenerate $\alpha$-derivation of
$L \otimes tK[t]/(t^n)$, and hence by Lemma \ref{l-d} $L \otimes tK[t]/(t^n)$ 
has a finite-dimensional faithful nilpotent nondegenerate (and multiplicative, 
if $L$ is multiplicative) representation; and so does its subalgebra $L$.
\end{proof}

Both conditions of nondegeneracy and zero characteristic are needed here to 
ensure nondegeneracy of $\alpha \otimes t\frac{\dcobound}{\dcobound t}$: the 
nondegeneracy of $\alpha$ ensures that it acts nondegenerately on the first 
tensor factor $L$, and the zero characteristic ensures that the Euler derivation
$t\frac{\dcobound}{\dcobound t}$ acts nondegenerately on the second tensor 
factor $tK[t]/(t^n)$: $t^i \mapsto it^i$ for $i\in \mathbb N$. We conjecture 
Lemma \ref{l-f} (in fact, the whole Ado theorem) remains to be true without 
those restrictions, but a proof of this will require a different approach (see 
discussion in \cite[\S 3]{ado} most of which may be applicable also to the 
Hom-Lie case).

\begin{lemma}\label{l-free}
Suppose the ground field is of characteristic zero. For any $k,n \in \mathbb N$,
and any polynomial $f$ with the nonzero free term:
\begin{enumerate}[\upshape(i)]
\item
the free nilpotent Hom-Lie algebra $\mathcal N_{k,n,f}$ has a finite-dimensional
faithful nilpotent nondegenerate representation;
\item\label{it-ii}
the free nilpotent multiplicative Hom-Lie algebra $\mathcal M_{k,n,f}$ has a 
finite-dimensional faithful nilpotent multiplicative nondegenerate 
representation.
\end{enumerate}
\end{lemma}

\begin{proof}
As noted in \S \ref{sec-free}, the algebras in question are finite-dimensional, 
$\mathbb N$-graded, and nondegenerate. Apply Lemma \ref{l-f}.
\end{proof}

The following lemma shows that instances of ``local faithfulness'' can be 
assembled to a ``global'' one.

\begin{lemma}\label{l-l}
Let $L$ be a finite-dimensional Hom-Lie algebra such that for any nonzero 
$x\in L$ there is a finite-dimensional (nilpotent, multiplicative, 
nondegenerate) representation $\rho_x$ of $L$ such that $\rho_x(x) \ne 0$. Then
$L$ has a finite-dimensional faithful (nilpotent, multiplicative, nondegenerate)
representation.
\end{lemma}

\begin{proof}
The proof repeats verbatim those of \cite[Lemma 2.7]{ado}. If each $\rho_x$ is
nilpotent, or multiplicative, or nondegenerate, the resulting representation,
being assembled as a direct sum of nilpotent, or multiplicative, or 
nondegenerate representations, is itself nilpotent, or multiplicative, or
nondegenerate.
\end{proof}

The following lemma, of combinatorial character, is an almost verbatim 
repetition of \cite[Lemma 2.10]{ado} which treats the Lie case, and shows that 
one can always distinguish elements of a finite-dimensional nilpotent Hom-Lie 
algebra by kernels of suitable representations.

\begin{lemma}\label{l-c}
Let $(L,\liebrack,\alpha)$ be a Hom-Lie algebra over a field $K$ of 
characteristic $\ne 2$, having a finite-dimensional faithful nilpotent 
multiplicative nondegenerate representation. Then for any two linearly 
independent elements $x,y \in L$ such that $\alpha(x) = \lambda x$ for some 
nonzero $\lambda \in K$, there is a finite-dimensional nilpotent multiplicative
nondegenerate representation $\rho$ of $L$ such that 
$\Ker \rho(x) \not\subset \Ker \rho(y)$.
\end{lemma}

\begin{proof}
Suppose the contrary: there are two linearly independent elements $x,y \in L$ 
such that $\alpha(x) = \lambda x$ for some nonzero $\lambda \in K$, and for any 
finite-dimensional nilpotent multiplicative nondegenerate representation 
$\rho: L \to \End(V)$, it holds that $\Ker \rho(x) \subseteq \Ker \rho(y)$. 
Then there is a linear map $h_\rho: V \to V$ (depending on $\rho$) such that 
\begin{equation}\label{eq-h}
\rho(y) = h_\rho \circ \rho(x) .
\end{equation}

Let $\rho: L \to \End(V)$ and $\tau: L \to \End(W)$ be two finite-dimensional
nilpotent multiplicative nondegenerate representations of $L$ in Hom-vector 
spaces $(V,\beta)$ and $(W,\gamma)$, and let $n$ and $m$ be the indices of 
nilpotency of the linear maps $\rho(x)$ and $\tau(x)$, respectively. According 
to Proposition~\ref{f-tens} and Lemma~\ref{lemma-tp}, the tensor product 
$\rho \otimes \tau$ is also multiplicative and nilpotent, and it is obviously 
nondegenerate. Writing the condition (\ref{eq-h}) for $\rho \otimes \tau$, with
both sides applied to elements 
$\rho(x)^{n-2}(v) \otimes \tau(x)^{m-1}\big(\gamma^{-1}(w)\big)$ and 
$\rho(x)^{n-1}\big(\beta^{-1}(v)\big) \otimes \tau(x)^{m-2}(w)$, where $v \in V$
and $w\in W$, and taking into account the same condition for $\rho$ and for 
$\tau$, we get respectively:
\begin{equation}\label{eq-d1}
h_\rho\Big(
\rho(x)^{n-1}(v)\Big) \otimes \gamma\Big(\tau(x)^{m-1}\big(\gamma^{-1}(w)\big)
\Big)
= 
h_{\rho \otimes \tau}\Big(
\rho(x)^{n-1}(v) \otimes \gamma\big(\tau(x)^{m-1}\big(\gamma^{-1}(w)\big)\big)
\Big)
\end{equation}
and
\begin{equation}\label{eq-d2}
\beta\Big(\rho(x)^{n-1}\big(\beta^{-1}(v)\big)\Big) \otimes 
h_\tau\Big(\tau(x)^{m-1}(w)\Big)
= 
h_{\rho \otimes \tau}\Big(
\beta\big(\rho(x)^{n-1}\big(\beta^{-1}(v)\big)\big) \otimes \tau(x)^{m-1}(w)
\Big) .
\end{equation}

Since $\alpha(x) = \lambda x$, and due to the condition of multiplicativity 
(\ref{eq-m}), we have 
\begin{equation}\label{eq-b}
\beta \circ \rho(x)^k = \lambda^k \rho(x)^k \circ \beta
\end{equation}
for any $k \in \mathbb N$, and similarly for $\tau$ and $\gamma$, so the 
equalities (\ref{eq-d1})--(\ref{eq-d2}) can be rewritten as
$$
h_\rho\Big(\rho(x)^{n-1}(v)\Big) \otimes \tau(x)^{m-1}(w)
= 
h_{\rho \otimes \tau}\Big(\rho(x)^{n-1}(v) \otimes \tau(x)^{m-1}(w)\Big)
$$
and
$$
\rho(x)^{n-1}(v) \otimes h_\tau\Big(\tau(x)^{m-1}(w)\Big)
= 
h_{\rho \otimes \tau}\Big(\rho(x)^{n-1}(v) \otimes \tau(x)^{m-1}(w)\Big)
$$
for any $v\in V$ and $w\in W$, respectively. This implies that the linear maps 
$h_\rho \otimes \id$ and $\id \otimes h_\tau$ coincide on the vector space
$\rho(x)^{n-1}(V) \otimes \tau(x)^{m-1}(W)$, whence
\begin{equation}\label{eq-mu}
h_\rho\Big(\rho(x)^{n-1}(v)\Big) = \mu \rho(x)^{n-1}(v)
\end{equation}
and
$$
h_\tau\Big(\tau(x)^{m-1}(w)\Big) = \mu \tau(x)^{m-1}(w)
$$
for some $\mu \in K$. Since this holds for any pair of representations $\rho$, 
$\tau$, we get that (\ref{eq-mu}) holds for any finite-dimensional nilpotent 
multiplicative nondegenerate representation $\rho$ of $L$ for some uniform value
of $\mu$.

Further, writing the condition (\ref{eq-h}) for the tensor product 
$\rho \otimes \tau$ applied to elements 
$\rho(x)^{n-3}\big(\beta(v)\big) \otimes \tau(x)^{m-1}\big(\gamma^{-1}(w)\big)$,
$\rho(x)^{n-1}\big(\beta^{-1}(v)\big) \otimes \tau(x)^{m-3}\big(\gamma(w)\big)$,
and $\rho(x)^{n-2}(v) \otimes \tau(x)^{m-2}(w)$, and taking into account 
(\ref{eq-mu}), we get respectively:
$$
h_\rho\Big(\rho(x)^{n-2}\big(\beta(v)\big)\Big) \otimes \tau(x)^{m-1}(w)
= 
h_{\rho \otimes \tau}\Big(\rho(x)^{n-2}\big(\beta(v)\big) \otimes \tau(x)^{m-1}(w)\Big)
,
$$
$$
\rho(x)^{n-1}(v) \otimes h_\tau\Big(\tau(x)^{m-2}\big(\gamma(w)\big)\Big)
= 
h_{\rho \otimes \tau}\Big(\rho(x)^{n-1}(v) \otimes \tau(x)^{m-2}\big(\gamma(w)\big)\Big)
,
$$
and
\begin{multline*}
\mu \Big(
\lambda^{m-2} \rho(x)^{n-1}(v) \otimes \tau(x)^{m-2}\big(\gamma(w)\big)
+ 
\lambda^{n-2} \rho(x)^{n-2}\big(\beta(v)\big) \otimes \tau(x)^{m-1}(w)
\Big)
\\= 
h_{\rho \otimes \tau} \Big(
\lambda^{m-2} \rho(x)^{n-1}(v) \otimes \tau(x)^{m-2}\big(\gamma(w)\big)
+
\lambda^{n-2} \rho(x)^{n-2}\big(\beta(v)\big) \otimes \tau(x)^{m-1}(w)
\Big) .
\end{multline*}

Taking a linear combination of the first two of these equalities with 
coefficients $\lambda^{n-2}$ and $\lambda^{m-2}$, and subtracting the third one,
we get that the linear map 
\begin{equation}\label{eq-lin}
\big((h_\rho - \mu\id) \circ \beta\big) \otimes \tau(x) 
+
\rho(x) \otimes \big((h_\tau - \mu\id) \circ \gamma\big)
\end{equation}
is identically zero on the vector space 
$\rho(x)^{n-2}(V) \otimes \tau(x)^{m-2}(W)$, whence 
$$
(h_\rho - \mu\id) \circ \beta = \eta_{\rho,\tau} \rho(x)
$$
and 
$$
(h_\tau - \mu\id) \circ \gamma = -\eta_{\rho,\tau} \tau(x)
$$
for some $\eta_{\rho,\tau} \in K$, as linear maps on $\rho(x)^{n-2}(V)$ and on 
$\tau(x)^{m-2}(W)$, respectively. Note that this holds for any pair of 
representations $\rho, \tau$. Taking $\rho = \tau$, we get 
$\eta_{\rho,\rho} = - \eta_{\rho,\rho}$, whence $\eta_{\rho,\rho} = 0$ (this is
the place where we need the assumption that characteristic of the ground field
is different from $2$), and
$$
h_\rho\Big(\beta\big(\rho(x)^{n-2}(v)\big)\Big) = 
\mu \beta\big(\rho(x)^{n-2}(v)\big)
$$
for any $v \in V$. Using here again (\ref{eq-b}), and taking $\beta^{-1}(v)$ 
instead of $v$, we get
$$
h_\rho\Big(\rho(x)^{n-2}(v)\Big) = \mu \rho(x)^{n-2}(v)
$$
for any finite-dimensional nilpotent multiplicative nondegenerate representation
$\rho$ of $L$\footnote{
The corresponding reasoning in the proof of \cite[Lemma 2.10]{ado} (p.~678 in 
the published version, and p.~4 in the arXiv version) is in error: at one place,
the sign went wrong -- the linear maps $(h_\rho - \lambda \id) \otimes \tau(x)$
and $\rho(x) \otimes (h_\tau - \lambda \id) \otimes \tau(x)$ do not coincide, 
but sum up to zero (compare with the expression (\ref{eq-lin}) here). The 
correct reasoning is obtained from one in the present paper by assuming all the
twist maps are equal to identity map. In fact, the correct reasoning is slightly
simpler than the original one, and avoids the necessity to assume non-vanishing
of certain binomial coefficients. This is significant, as it eliminates one of 
the places where the characteristic zero assumption is needed (see the 
discussion in \cite[\S 3]{ado}).
}.

Repeating this procedure by considering at the $k$th step the condition 
(\ref{eq-h}) for the tensor product of two representations $\rho$ and $\tau$, 
and applying it to all elements of the form $\rho(x)^i(v) \otimes \tau(x)^j(w)$
with $i+j$ equal to $n+m-(k+2)$, where $n$ and $m$ are the indices of nilpotency
of $\rho(x)$ and $\tau(x)$ respectively, and heavily using (\ref{eq-b}), we 
consecutively arrive at the equalities
$$
h_\rho\Big(\rho(x)^{n-k}(v)\Big) = \mu \rho(x)^{n-k}(v)
$$
for each $k=1,2,\dots,n$. For $k=n$ this means $h_\rho = \mu \id$, and hence
$\rho(y - \mu x) = 0$ for any finite-dimensional nilpotent multiplicative
nondegenerate representation $\rho$ of $L$, what implies $y - \mu x = 0$, a 
contradiction.
\end{proof}

It seems that the conclusion of Lemma~\ref{l-c} is valid for arbitrary linearly
independent elements $x,y$ of the Hom-Lie algebra, without the assumption 
$\alpha(x) = \lambda x$. However, this assumption greatly simplifies the 
calculations, and we will not venture into attempting to give a proof without 
it; Lemma~\ref{l-c} as stated is enough for our purposes here.

Finally, we arrive at our main result, whose proof is assembled from the 
previous lemmas exactly in the same way as the proof of 
\cite[Theorem 2.11]{ado}.

\begin{theorem}\label{th-ado}
A finite-dimensional nilpotent multiplicative nondegenerate Hom-Lie algebra over
an algebraically closed field of characteristic zero has a finite-dimensional 
faithful nilpotent multiplicative nondegenerate representation.
\end{theorem}

\begin{proof}
Let us present a finite-dimensional nilpotent multiplicative nondegenerate 
Hom-Lie algebra $L$ defined over an algebraically closed field $K$ of 
characteristic zero, as a quotient $\mathcal M_{k,n,f}/I$ for suitable 
$k,n \in \mathbb N$ and polynomial $f$ (with the nonzero free term). We will 
proceed by induction on the dimension of $I$. The case $I = 0$ is covered by 
Lemma~\ref{l-free}(\ref{it-ii}).

Suppose $I$ is nonzero. By Proposition~\ref{prop-nilp} $\mathcal M_{k,n,f}$ is
strongly nilpotent, and by Lemma~\ref{lemma-ij} there is an ideal $J$ of 
$\mathcal M_{k,n,f}$ such that $J \subset I$, $\dim I/J = 1$, and 
$[\mathcal M_{k,n,f},I] \subseteq J$. Consequently, 
$\widetilde L = \mathcal M_{k,n,f}/J$ is an extension of $L$ by an 
one-dimensional central ideal, say, $K\tilde{z}$. Since $K\tilde{z}$ is an ideal
in the Hom-Lie algebra $(\widetilde L, \liebrack, \alpha)$, we have 
$\alpha(\tilde{z}) = \lambda \tilde{z}$ for some $\lambda \in K$; and since 
$\mathcal M_{k,n,f}$ is finite-dimensional and nondegenerate, so is 
$\widetilde L$, what implies $\lambda \ne 0$.

Take an arbitrary nonzero $x\in L$, and consider its preimage $\tilde{x}$ in 
$\widetilde L$. By the induction assumption, $\widetilde L$ has a 
finite-dimensional faithful nilpotent multiplicative nondegenerate 
representation, and by Lemma~\ref{l-c} there is a finite-dimensional nilpotent
multiplicative nondegenerate representation $\rho: \widetilde L \to \End(V)$ 
such that 
\begin{equation}\label{eq-k}
\Ker\rho(\tilde{z}) \not\subset \Ker\rho(\tilde{x}) .
\end{equation}

Since $\tilde{z}$ lies in the center of $\widetilde L$, for any 
$\tilde y \in \widetilde L$ we have 
$$
\rho(\alpha(\tilde y)) \circ \rho(\tilde z) = 
\lambda \rho(\tilde z) \circ \rho(\tilde y) ,
$$
and hence the vector space $\Ker \rho(\tilde z)$ is an $\widetilde L$-submodule
of $V$, and thus also carries a natural structure of an $L$-module on which $x$
acts nontrivially, due to (\ref{eq-k}). Since the initial representation $\rho$
of $\widetilde L$ is nilpotent, multiplicative, and nondegenerate, the ensuing 
representation $\tau$ of $L$ is nilpotent, multiplicative, and nondegenerate 
too.

Thus we see that for any nonzero $x\in L$ there is a finite-dimensional 
nilpotent multiplicative nondegenerate representation $\tau$ of $L$ such that 
$\tau(x) \ne 0$, and by Lemma~\ref{l-l} $L$ has a finite-dimensional faithful 
nilpotent multiplicative nondegenerate representation.
\end{proof}

\section*{Acknowledgements}

This work was started during the visit of the second author to Universit\'e de 
Haute-Alsace.

\end{document}